\documentclass[12pt,a4paper]{amsart}
\usepackage[margin=3cm]{geometry}

\usepackage[utf8]{inputenc}
\usepackage[english]{babel}
\usepackage{amsmath,amssymb,amsthm}
\usepackage{mathrsfs}
\usepackage{tcolorbox}
\usepackage{xspace}
\usepackage{booktabs}
\usepackage[shortlabels]{enumitem}
\usepackage{pgf}
\usepgflibrary{shapes.arrows}
\usepackage{pgfplots}
\pgfplotsset{compat=newest}
\usepackage{listings}

\newtheorem{theorem}{Theorem}[section]

\newtheorem{lemma}[theorem]{Lemma}
\newtheorem{proposition}[theorem]{Proposition}
\newtheorem{definition}[theorem]{Definition}

\theoremstyle{definition}

\newtheorem{remark}[theorem]{Remark}

\newcommand{\PGU}{\mathsf{PGU}}
\newcommand{\PG}{\mathsf{PG}}
\newcommand{\GF}{\mathsf{GF}}
\newcommand{\AGL}{\mathsf{AGL}}

\DeclareMathOperator{\Persp}{Persp}

\newcommand{\BBT}{\textbf{BBT}\xspace}
\newcommand{\KRC}{\textbf{KRC}\xspace}
\newcommand{\KNP}{\textbf{KNP}\xspace}
\newcommand{\BB}{\textbf{BB}\xspace}

\title{On the geometry of full points of abstract unitals}
\author{D\'avid Mez\H{o}fi}
\address{Bolyai Institute \\
        University of Szeged \\
        Aradi v\'ertan\'uk tere 1\\
        H-6720 Szeged, Hungary}
\email{mezofi@math.u-szeged.hu}

\author{G\'abor P. Nagy}
\address{Department of Algebra \\
        Budapest University of Technology and Economics\\
        Egry J\'ozsef utca 1\\
        H-1111 Budapest, Hungary}
\address{Bolyai Institute \\
        University of Szeged \\
        Aradi v\'ertan\'uk tere 1\\
        H-6720 Szeged, Hungary}
\email{nagyg@math.bme.hu}

\thanks{Support provided from the National Research, Development and Innovation Fund of Hungary, financed under the 2018-1.2.1-NKP funding scheme, within the SETIT Project 2018-1.2.1-NKP-2018-00004. Partially supported by OTKA grants 119687 and 115288.}

\keywords{Abstract unital, projective embedding, perspectivities, full point}
\subjclass[2010]{51E20, 05B25}

\begin{document}

\begin{abstract}
The concept of full points of abstract unitals has been introduced by Korchm\'aros, Siciliano and Sz\H{o}nyi as a tool for the study of projective embeddings of abstract unitals. In this paper we give a more detailed description of the combinatorial and geometric structure of the sets of full points in abstract unitals of finite order.
\end{abstract}

\maketitle

\section{Introduction}

An abstract unital of order $n$ is a $2$-$(n^3+1,n+1,1)$ design. We say that an abstract unital $(X,B)$ is embedded in a projective plane $\Pi$ if $X$ consists of points of $\Pi$ and each block $b \in B$ has the form $X\cap \ell$ for some line $\ell$ of $\Pi$. For results on projective embeddings of abstract unitals see \cite{KSSz2018} and the references therein. 

Let $U=(X,B)$ be an abstract unital of order $n$ and fix two blocks $b_1,b_2$. Using the terminology of \cite{KSSz2018}, we say that $P$ is a \textit{full point with respect to} $(b_1,b_2)$ if $P\not\in b_1 \cup b_2$ and for each $Q\in b_1$, the block connecting $P$ and $Q$ intersects $b_2$. In other words, there is a well defined projection $\pi_{b_1,P,b_2}$ from $b_1$ to $b_2$ with center $P$. We denote by $F_U(b_1,b_2)$ the set of full points of $U$ w.r.t. the blocks $b_1,b_2$. Clearly, $F_U(b_1,b_2)=F_U(b_2,b_1)$. 

The structure of the paper is as follows. The main result of this paper is proved in Section 3. It shows that for any abstract unital of order $q$, which is projectively embedded in the Galois plane $\PG(2,q^2)$, the set of full points of two disjoint blocks are contained in a line. Moreover, the perspectivities of two disjoint blocks generate a semi-regular cyclic permutation group. In Section 4, we extend the results of \cite{KSSz2018} by giving a complete description on the structure of full points in the classical Hermitian unitals. Section 5 gives an overview of computational results about full points in abstract unitals of order $3$ and $4$, which belong to known classes \cite{BagchiBagchi1989,BBT03,KNP11,Kr06}. For the computation we developed and used the GAP package \texttt{UnitalSZ} \cite{MezNagy2018}.

\section{Combinatorial properties of the set of full points}

\subsection{Bounds on the number of full points}

We start with an easy observation on the number of full points of two blocks $b_1,b_2$ of $U$. The result seems to be rather weak. 

\begin{lemma}  \label{lem:generic_bounds}
Let $U=(X,B)$ be an abstract unital of order $n\geq 2$. Then
\[|F_U(b_1,b_2)|\leq 
\begin{cases}
n^2-n & \text{if $b_1,b_2$ have a point in common,} \\
n^2-1 & \text{if $b_1,b_2$ are disjoint.} 
\end{cases} \]
\end{lemma}
\begin{proof}
For a fixed point $P\in b_1$ we define the set $S'_P$ as the union of the blocks connecting $P$ with $Q\in b_2\setminus b_1$, and the set $S_P=S'_P\setminus (b_1\cup b_2)$. Clearly, 
\[|S_P|=\begin{cases}
n^2-n & \text{if $b_1,b_2$ have a point in common,} \\
n^2-1 & \text{if $b_1,b_2$ are disjoint.} 
\end{cases}\]
As $F_U(b_1,b_2) \subseteq S_P$, the lemma follows. 
\end{proof}

In most (but not all) known examples of abstract unitals, the set of full points is contained in a block. This motivates the following definition. 

\begin{definition}
Let $U=(X,B)$ be an abstract unital and $b_1,b_2 \in B$ disjoint blocks. 
\begin{enumerate}[(i)]
\item The triple $\left( U, b_{1}, b_{2} \right)$ is \textit{full point regular} if the set of full points $F_{U}{\left( b_{1}, b_{2} \right)} \subseteq c$ for some block $c\in B$ such that $b_{1} \cap c = b_{2} \cap c = \emptyset$.
\item The abstract unital $U$ is \textit{full point regular} if for any two disjoint blocks $b_{1}, b_{2}$ the triple $\left( U, b_{1}, b_{2} \right)$ is full point regular.
\end{enumerate}
\end{definition}


\subsection{Full points and perspectivities} By definition, any full point $P$ of the blocks $b_1,b_2$ defines a bijective map $\pi_{b_1,P,b_2}:b_1\to b_2$; we call it the \textit{perspectivity with center $P$.}

\begin{definition}
Let $b_1,b_2$ be blocks of the abstract unital $U$. Define the \textit{group of perspectivities of} $b_1$ as
\[\Persp_{b_2}(b_1) = \langle \pi_{b_1,P,b_2} \pi_{b_2,Q,b_1} \mid P,Q\in F_U(b_1,b_2)\rangle. \]
\end{definition}

It is easy to see that $\Persp_{b_2}(b_1)$ and $\Persp_{b_1}(b_2)$ are isomorphic permutation groups, the former acting on $b_1$ and the latter acting on $b_2$. For different full points $Q,R$, the perspectivities $\pi_{b_1,Q,b_2}$ and $\pi_{b_1,R,b_2}$ are different. This implies $|\Persp_{b_2}(b_1)| \geq |F_U(b_1,b_2)|$. In particular, $\Persp_{b_2}(b_1)$ is nontrivial if $|F_U(b_1,b_2)|> 1$. An important case will be when $\Persp_{b_2}(b_1)$ is a cyclic semi-regular permutation group on $b_1$.


\subsection{Dual $k$-nets in abstract unitals} We will present examples of abstract unitals when the set of full points w.r.t. the blocks $b_1,b_2$ form a third block $b_3$. More generally, we introduce the concept of an embedded dual $k$-net of an abstract unital. An abstract $k$-net is a structure consisting of a set $X$ of points and a set $B$ of blocks, which is partitioned into $k$ disjoint families $B_1,\ldots,B_k$ for which the following hold: (1) every point is incident with exactly one block of every $B_i$, $(i = 1,\ldots, k)$; (2) two lines of different families have exactly one point in common; (3) there exist $3$ lines belonging to $3$ different $B_i$, and which are not incident with the same point. See \cite{BarlottiStrambach,Belousov} as reference on abstract $k$-nets.

\begin{definition}
Let $U=(X,B)$ be an abstract unital of order $n$ and $k\geq 3$ and integer. We say that the blocks $b_1,\ldots,b_k$ form an \textit{embedded dual $k$-net} in $U$, if the following hold for all $1\leq i<j\leq k$:
\begin{enumerate}[(i)]
\item $b_i\cap b_j=\emptyset$.
\item For all $P\in b_i$, $Q\in b_j$, the block containing $P,Q$ intersects all $b_1,\ldots,b_k$ in a point. 
\end{enumerate}
\end{definition}

It is clear that for an embedded dual $k$-net $b_1,\ldots,b_k$ of $U$, $b_3\cup \cdots \cup b_k \subseteq F_U(b_1,b_2)$. The converse needs some explanation. 

\begin{lemma} \label{lem:dual_k_nets}
Let $U$ be an abstract unital of order $n$, $k\geq 3$ an integer and $b_1,\ldots,b_k$ blocks of $U$.
\begin{enumerate}[(i)]
\item If $b_3\subseteq F_U(b_1,b_2)$, then $b_1$ and $b_2$ are disjoint. 
\item If $b_3\subseteq F_U(b_1,b_2)$, then $b_1\subseteq F_U(b_2,b_3)$ and $b_2\subseteq F_U(b_1,b_3)$.
\item If $b_3\cup b_4\subseteq F_U(b_1,b_2)$, then $b_3$ and $b_4$ are disjoint. 
\item The blocks $b_1,\ldots,b_k$ form an embedded dual $k$-net if and only if $b_3\cup \cdots \cup b_k \subseteq F_U(b_1,b_2)$. 
\end{enumerate}
\end{lemma}
\begin{proof}
(i) Assume that $\{Z\} = b_1\cap b_2$ and $b_3\subseteq F_U(b_1,b_2)$. Clearly, $b_3$ is disjoint from $b_1\cup b_2$. Fix an arbitrary point $P\in b_1\setminus\{Z\}$. Each point $R\in b_3$ projects $P$ to $b_2\setminus\{Z\}$. Hence, there are points $R_1,R_2 \in b_3$ such that $\pi_{b_1,R_1,b_2}(P)=\pi_{b_1,R_2,b_2}(P)$. This means that $P\in b_3$, a contradiction. (ii)  For any $P_1\in b_1$, $P_3\in b_3$, the block $P_1P_3$ intesects $b_2$. Now fix $P_1$ and let $P_3$ run through $b_3$ in order to obtain the bijection $\pi_{b_3,P_1,b_2}$. Thus, $P_1 \in F_U(b_2,b_3)$. Since this holds for all $P_1\in b_1$, the claim follows. For (iii) it suffices to show $b_1\subseteq F_U(b_3,b_4)$. Take $P \in b_1$, $Q\in b_3$ arbitrary points. From $Q$, $P$ projects to $R\in b_2$ and using $b_2\subseteq F_U(b_1,b_4)$, $P$ projects to $S\in b_4$ from $R$. Hence, $Q$ projects to $b_4$ from $P$. 

The ``only if'' part of (iv) follows from the definition. Assume now $b_3\cup \cdots \cup b_k \subseteq F_U(b_1,b_2)$. By (i) and (iii), all blocks $b_1,\ldots,b_k$ are disjoint. For the indices $3\leq i<j\leq k$,  there is an injective map $\alpha \mid b_1 \times b_2 \rightarrow b_i \times b_j$ mapping $\left( P_1,
P_2 \right) \mapsto \left( P_i, P_j \right)$ with collinear quadruple $P_1,P_2,P_i,P_j$. Moreover $\alpha$ is bijective, hence any pair of points $\left( P_i, P_j \right) \in b_i \times b_j$ determines a block $b'$ of $U$ such that $b'\cap b_i=P_i$, $i=1,2$. The block joining $P_1$ and $P_2$ intersects any block $b_s \subseteq F_{U}{\left( b_1, b_2 \right)}$ in $P_s$ for $3 \leq s \le k$, therefore $b_1,\ldots,b_k$ form a dual $k$-net in $U$.\end{proof}

\subsection{Bounds on dual $k$-nets in abstract unitals} For embedded dual $k$-nets, the trivial bound is $k\leq n+1$. With some elementary counting, we can improve this to $k\leq n-1$. This implies that an abstract unital of order $3$ has no embedded dual $3$-nets. 

\begin{proposition} \label{prop:dual_k_nets}
Let $U$ be an abstract unital of order $n\geq 3$. 
\begin{enumerate}[(i)]
\item If $U$ has an embedded dual $k$-net $\{b_1,\ldots,b_k\}$, then $k\leq n-1$. 
\item For two blocks $b_1,b_2$, $F_U(b_1,b_2)$ cannot contain more than $n-3$ blocks. 
\end{enumerate}
\end{proposition}
\begin{proof}
(i) Let us assume that $k > n - 1$ and let $\mathcal{P} = b_1 \cup b_2 \cup
\ldots \cup b_k$. Any block of $U$ intersects $\mathcal{P}$ in $0$, $1$, $k$ or
$n+1$ points, the latter being the blocks $b_i$ themselves. W.l.o.g. consider the
disjoint blocks $b_1, b_2$. Any pair of points chosen from $b_1$ and $b_2$
determines the unique block in $B$ which is a $k$-secant to $\mathcal{P}$,
therefore the number of $k$-secants is $\left( n + 1 \right)^{2}$. Then, fix an
arbitrary block $b_{i}$ of the dual $k$-net and a point $P$ on the block
$b_{i}$. The number of $1$-secant blocks on $P$ is $n^{2} - n - 2$. Thus the
number $1$-secant blocks to $\mathcal{P}$ is $k\left( n + 1 \right)\left( n^{2}
- n - 2 \right)$. Since $\left| B \right| = n^{2}\left( n^{2} - n + 1 \right)$
we have
\begin{align*}
    k + \left( n + 1 \right)^{2} + k\left( n + 1 \right)\left( n^{2} - n - 2
    \right) &\le n^{2}\left( n^{2} - n + 1 \right),
\end{align*}
which gives $n^{3} - 3n^{2} + n + 1 \le 0$ by $k\geq n\geq 3$, a contradiction. 

(ii) If $F_U(b_1,b_2)$ contains the $k-2$ blocks $b_3,\ldots,b_k$, then $\{b_1,\ldots,b_k\}$ is an embedded dual $k$-net in $U$ by Lemma \ref{lem:dual_k_nets}(iv). Hence, $k-2\leq n-3$ by (i). 
\end{proof}

\subsection{Embedded dual $3$-nets and latin squares}

An embedded dual $3$-net $\{b_1,b_2,b_3\}$ determines a latin square $L$ of order $n+1$ in the following way. Label the points of $b_1, b_2, b_3$ by the set $\{1,\ldots,n+1\}$: 
\[b_s=\{P_{s,1}, \ldots,P_{s,n+1}\}, \qquad s=1,2,3. \] 
For $i,j\in \{1,\ldots,n+1\}$, let $c$ be the block connecting $P_{1,i}$ and $P_{2,j}$. Define $s$ by $\{P_{3,s}\} = b_3\cap c$ and write $s$ in row $i$ and column $j$ of $L$. Choosing a different labeling for $b_1,b_2,b_3$ results in an \textit{isotope} latin square. By reordering the three blocks, one gets \textit{conjugate} or \textit{parastrophe} latin squares. The set of all parastrophes of a latin square $L$ is also called the \textit{main class} of $L$. Latin squares are naturally related to (the multiplication tables of) finite \textit{quasigroups.} See \cite[Section 1.4]{DenesKeedwell} for more details and further references on conjugacy and parastrophy of latin squares. 

A property which, for each class $C$, either holds for all members of $C$ or for no member of $C$ is said to be a \textit{class invariant.} Properties of the underlying (dual) $3$-nets are \textit{main class invariants} of the corresponding coordinate latin square. In particular, the groups of perspectivities can be defined for (dual) $3$-nets and they are useful examples of \textit{main class invariants.} In the primal setting, perspectivities of 3-nets have been presented in \cite{BarlottiStrambach} and \cite{Belousov}. 

Let $L$ be a latin square of order $n$. We say that $L$ is group-based if it is a parastrophe to the Cayley table of a group $G$ of order $n$. As the group $G$ only depends on the main class of $L$, the following concept is well-defined.

\begin{definition}
Let $\mathcal{B}=\{b_1,b_2,b_3\}$ be an embedded dual $3$-net of the abstract unital $U$. We say that $\mathcal{B}$ is cyclic, if the corresponding latin square is a parastrophe of the Cayley table of the cyclic group of order $n+1$, where $n$ is the order of $U$.
\end{definition}

\begin{proposition} \label{prop:cyclic_dual_nets}
Let $U$ be an abstract unital of order $n$ and $\mathcal{B}=\{b_1,b_2,b_3\}$ be an embedded dual $3$-net of $U$. The following are equivalent:
\begin{enumerate}[(i)]
\item $\mathcal{B}$ is cyclic.
\item $\Persp_{b_i}(b_j)$ is the cyclic group of order $n+1$ for all $1\leq i,j\leq 3$, $i\neq j$. 
\end{enumerate}
\end{proposition}
\begin{proof}
Let $L$ be the latin square associated to $\mathcal{B}$. By \cite[Proposition 1.2]{BarlottiStrambach}, (ii) implies that the rows of $L$ are elements of the cyclic group of order $n$, hence $L$ is cyclic and (i) holds. Conversely, assume that $\mathcal{B}$ is labeled in such a way that the the coordinate latin square $L$ is the Cayley table of the cyclic group. Then \cite[Theorem 6.1]{BarlottiStrambach} implies (ii). 
\end{proof}

\section{Full point regularity of embedded unitals} \label{sec:fpr}

The questions on the embeddings of abstract unitals in projective planes are long studied problems, with special focus on the embeddings of abstract unitals of order $q$ in the desarguesian plane $\PG(2,q^2)$. Korchm\'aros, Siciliano and Sz\H{o}nyi \cite{KSSz2018} developed the method of full points for the embedding problem. The main tool is the group of perspectivities of unital blocks. We notice that while the permutation group $\Persp_{b_2}(b_1)$ depends only on the abstract unital structure of $U=(X,B)$, we may be able compute it more easily when a projective embedding of $U$ is given.

Although the definition of the group of perspectivities works for intersecting blocks $b_1,b_2$, in the sequel, we will only deal with the case when $b_1,b_2$ are disjoint. The next definition gives a stronger version of the full point regular property, using the structure of the group of perspectivities. 

\begin{definition}
Let $U=(X,B)$ be an abstract unital and $b_1,b_2 \in B$ disjoint blocks. 
\begin{enumerate}[(i)]
\item If $\left( U, b_{1}, b_{2} \right)$ is a full point regular triple and $\Persp_{b_2}(b_1)$ is a cyclic semi-regular permutation group of $b_1$, then $\left( U, b_{1}, b_{2} \right)$ is said to be \textit{strongly full point regular.}
\item The abstract unital $U$ is \textit{strongly full point regular} if for any two disjoint blocks $b_{1}, b_{2}$ the triple $\left( U, b_{1}, b_{2} \right)$ is strongly full point regular.
\end{enumerate}
\end{definition}

Notice that $U$ is strongly full point regular if it has no full points at all. The next two lemmas deal with elementary properties of the groups of affinities of projective lines in $\PG(2,q^2)$, where $q$ is a power of the prime $p$. 

\begin{lemma} \label{lem:agl1-prop}
Let $p$ be a prime. 
\begin{enumerate}[(i)]
\item Let $g$ be an element of the affine linear group $\AGL(1,p^f)$ such that $o(g) \mid p^f-1$. Then $g$ has a unique fixed point $v\in \mathbb{F}_{p^f}$ and permutes $\mathbb{F}_{p^f}$ in orbits of length $o(g)$. 
\item Let $S$ be a subgroup of $\AGL(1,p^f)$ such that $p\nmid |S|$. Then, $S$ is cyclic and $|S|$ divides $p^f-1$. Moreover, $S$ has a unique fixed point in $\mathbb{F}_{p^f}$. \qed
\end{enumerate}
\end{lemma}

\begin{lemma}\label{lem:persp-elm}
  Let $\ell_1,\ell_2$ be two lines of $\PG(2,q^2)$ and $P,Q$ be two points off
  $\ell_1 \cup \ell_2$. Write $Z=\ell_1 \cap \ell_2$ and $V_i=\ell_i\cap PQ$,
  $i=1,2$. The perspectivity $\pi_{\ell_1,P,\ell_2} \pi_{\ell_2,Q,\ell_1}$
  fixes $Z$ and $V_1$ and permutes $\ell_1\setminus \{Z,V_1\}$ in orbits of
  equal lengths. 
\end{lemma}
\begin{proof}
Elementary. 
\end{proof}

Let $S$ be any set of $n+1$ points in the projective plane $\Pi$ of order $n$. A \textit{nucleus} of $S$ is a point $P$ such that each line of $\Pi$ through $P$ intersects $S$ in a unique point. It follows that $P\not\in S$. We denote by $\mathcal{N}(S)$ the set of all nuclei of $S$. 

Let $U = \left( X, B \right)$ be a unital of order $q$ embedded in
$\PG{\left( 2, q^{2} \right)}$ and let $b_1, b_2 \in B$ be two (not
necessarily disjoint) blocks of $U$. Denote the lines containing the blocks
$b_1$ and $b_2$ by $\ell_1$ and $\ell_2$ respectively. Using the notations in
\cite{KM94} let $\mathcal{B} = b_1 \cup \left( \ell_2 \setminus b_2 \right)$:
the set $\mathcal{B}$ consists of $q^{2} + 1$ non collinear points, it is
contained in the union of the lines $\ell_1$ and $\ell_2$. Note that $Z =
\ell_1 \cap \ell_2$ belongs to $\mathcal{B}$. Let $\mathcal{N}{\left(
\mathcal{B} \right)}$ denote the set of all nuclei of $\mathcal{B}$. Clearly,
if $P$ is a full point w.r.t. to the blocks $b_1, b_2$ then $P$ is a nucleus
of $\mathcal{B}$, hence $F_{U}{\left( b_1, b_2 \right)} \subseteq
\mathcal{N}{\left( \mathcal{B} \right)}$.

The next lemma formulates \cite[Propositions~2 and~3]{KM94} in our setting.

\begin{lemma}
  Let $U = \left( X, B \right)$ be a unital of order $q$ embedded in
  $\PG{\left( 2, q^{2} \right)}$ and let $b_1, b_2 \in B$ be two blocks of $U$.
  Denote the lines containing the blocks $b_1$ and $b_2$ by $\ell_1$ and
  $\ell_2$ respectively. Write $Z=\ell_1 \cap \ell_2$ and $\mathcal{B} = b_1
  \cup \left( \ell_2 \setminus b_2 \right)$. Define the set $\Gamma_1 =
  \left\{ \pi_{\ell_1,P,\ell_2} \pi_{\ell_2,Q,\ell_1} \mid P,Q \in
  \mathcal{N}{\left( \mathcal{B} \right)} \right\}$ where $\mathcal{N}{\left(
  \mathcal{B} \right)}$ denotes the set of all nuclei of $\mathcal{B}$. Then
  the following hold:
  \begin{enumerate}[(i)]
    \item $\Gamma_1$ leaves $b_1$ invariant.
    \item $\Gamma_1$ is a group of affinities of the affine line $\ell_1
      \setminus \left\{ Z \right\}$. \qed
  \end{enumerate}
\end{lemma}

Define the integer $r$ by $q^2=p^r$. The order of the group $\Gamma_1$ is $tp^{h}$, where $p \nmid t$, and
$\Gamma_{1}$ is isomorphic to some group $\Gamma = \mathbf{A}\mathbf{B}$ of
affinities where $\mathbf{B}$ is an additive subgroup of order $p^{h}$ of
$\GF{\left( q^{2} \right)}$ and $\mathbf{A}$ is a multiplicative subgroup of
order $t$ of $\GF{\left( q^{2} \right)}$ such that $t \mid p^{\gcd\left( r, h
\right)} - 1$. Let $m = \left( p^{r - h} - 1 \right) / t$ and let
$\mathbf{B}_{1} \cup \mathbf{O}_{1} \cup \ldots \cup \mathbf{O}_{m}$ be the
partition of $\ell_{1} \setminus \left\{ Z \right\}$ into $\Gamma_1$-orbits. We
have by \cite[Section 2]{KM94} that $\mathbf{B}_{1}$ has length $p^{h}$ and for
each $i = 1, 2, \ldots, m$ the orbit $\mathbf{O}_{i}$ has length $tp^{h}$.

Let $\mathcal{B}_{i} = \ell_{i} \cap \mathcal{B}$ for $i = 1, 2$ and let
$\widehat{\mathcal{B}}_{1} = \mathcal{B}_{1} \setminus \left\{ Z \right\}$,
then $\widehat{\mathcal{B}}_{1}$ is the union of $\Gamma_{1}$-orbits. It
follows that the size of $\widehat{\mathcal{B}}_{1}$ must be divisible by
$p^{h}$, and we must distinguish between two cases:
\begin{enumerate}
  \item If the blocks $b_{1}$ and $b_{2}$ are disjoint, it means $b_{1} =
    \mathcal{B}_{1} = \widehat{\mathcal{B}}_{1}$, hence $p^{h} \mid q + 1$. It
    is possible only for $h = 0$, thus the group $\mathbf{B}$ is trivial.
  \item Otherwise $b_{1} \cap b_{2} = \left\{ Z \right\}$, meaning $b_{1} =
    \mathcal{B}_{1} = \widehat{\mathcal{B}}_{1} \cup \left\{ Z \right\}$, hence
    the size of $\widehat{\mathcal{B}}_{1}$ is $q$. In this case $q = ap^{h} +
    btp^{h}$, where $b \in \left\{ 0, 1, \ldots, m \right\}$ and $a = 1$ or
    $0$, depending on whether $\mathbf{B}_{1} \subseteq
    \widehat{\mathcal{B}}_{1}$ or not. If $a = 0$, then $q = btp^{h}$, and as
    $p \nmid t$ we have $t = 1$, therefore the group $\mathbf{A}$ is trivial.
\end{enumerate}

\begin{lemma}\label{lem:gamma1-has-fixedpoint}
  Let $U = \left( X, B \right)$ be a unital of order $q$ embedded in
  $\PG{\left( 2, q^{2} \right)}$ and let $b_1, b_2 \in B$ be two disjoint
  blocks of $U$. Denote the lines containing the blocks $b_1$ and $b_2$ by
  $\ell_1$ and $\ell_2$ respectively. Write $Z=\ell_1 \cap \ell_2$ and
  $\mathcal{B} = b_1 \cup \left( \ell_2 \setminus b_2 \right)$. Define the
  group $\Gamma_1$ generated by the perspectivities $\pi_{\ell_1,P,\ell_2}
  \pi_{\ell_2,Q,\ell_1}$ with $P,Q \in \mathcal{N}{\left( \mathcal{B} \right)}$
  where $\mathcal{N}{\left( \mathcal{B} \right)}$ denotes the set of all nuclei
  of $\mathcal{B}$. Then the following hold:
  \begin{enumerate}[(i)]
    \item $p \nmid \left| \Gamma_1 \right|$. 
    \item $\Gamma_1$ is cyclic and $|\Gamma_1| \mid q^2-1$. 
    \item $\Gamma_1$ has a unique fixed point $V_1 \not\in b_1 \cup \left\{
      Z \right\}$. 
    \item The set of full points $F_{U}{\left( b_1, b_2 \right)}$ is contained
      in a line $m$ through $V_1$ but $Z$.
  \end{enumerate}
\end{lemma}
\begin{proof}
  Assume that $\Gamma_1$ has an element $\gamma$ of order $p$.  Since $b_{1}$
  is $\Gamma_1$-invariant, $\gamma$ has a fixed point in $b_{1}$, different
  from $Z$ as $Z \not\in b_1$. However, affinities with two fixed points
  have order dividing $q^2-1$. This proves (i).
  
  Together with Lemmas~\ref{lem:agl1-prop} and~\ref{lem:persp-elm}, (i)
  implies (ii) and (iii). Notice that Lemma \ref{lem:agl1-prop}(i) is needed to
  show that $V_1\not\in b_1$.
  
  Since $\mathbf{B}$ is trivial, the set of nuclei $\mathcal{N}{\left(
  \mathcal{B} \right)}$ is contained in a line $m$ such that $Z \not\in m$
  (cf.~\cite[p.~67]{KM94}). In particular $F_{U}{\left( b_1, b_2 \right)}$ is
  contained in $m$ as $F_{U}{\left( b_1, b_2 \right)} \subseteq
  \mathcal{N}{\left( \mathcal{B} \right)}$. Furthermore, by Lemma
  \ref{lem:persp-elm}, for any $P,Q\in \mathcal{N}{\left( \mathcal{B} \right)}$
  the line $PQ$ contains $V_1$, hence $V_1 \in m$. This proves (iv).
\end{proof}

We can now state and prove the main theorem of this section. 

\begin{theorem}\label{thm:embed-fptregular}
  If the unital $U$ of order $q$ is embedded in $\PG{\left( 2, q^{2} \right)}$
  then it is strongly full point regular.
\end{theorem}
\begin{proof}
  Let us assume that $U$ is embedded in $\PG(2,q^2)$. Let $b_1,b_2$ be two
  disjoint blocks of $U$. If $|F_U(b_1,b_2)|\leq 1$ then we have nothing to
  prove. Otherwise, by Lemma~\ref{lem:gamma1-has-fixedpoint} $F_U(b_1,b_2)$ is
  contained in a block $c$, which is disjoint to $b_1$ and $b_2$. Furthermore,
  $\Persp_{b_2}(b_1)$ is cyclic, its order divides $q^2-1$ and $b_1$ decomposes
  into orbits of equal lengths. This means that $(U,b_1,b_2)$ is a strongly
  full point regular triple. 
\end{proof}

\section{Full points of the Hermitian unital}

For a prime power $q$, let $\rho$ be a Hermitian polarity of $\PG(2,q^2)$. Two
points $P,Q$ are said to be \textit{conjugate} if $P\in Q^\rho$. Similarly, the
lines $\ell,m$ are \textit{conjugate} if $\ell^\rho \in m$. Let $R^+$ be the set
of pairs $(\ell,m)$, where $\ell,m$ are conjugate lines to each other but not
self-conjugate. The projective unitary group $\PGU(3,q)$ acts transitively on
$R^+$. Given two conjugate lines $\ell_1,\ell_2$, one constructs
$\ell_3=(\ell_1\cap \ell_2)^\rho$, conjugate to both
$\ell_1$ and $\ell_2$. We say that $\ell_1,\ell_2,\ell_3$ form a \textit{polar
triangle.} The projective unitary group $\PGU(3,q)$ acts transitively on the set
of polar triangles. Consider the set ${X}$ of self-conjugate points of $\rho$;
$|{X}|=q^3+1$. The line $\ell$ intersects ${X}$ in $1$ or $q+1$ points,
depending on if $\ell$ is self-conjugate or not. Let $\ell$ be a non
self-conjugate line and $m$ be a line connecting $\ell^\rho$ and a point $P\in
X\cap \ell$. Since $\ell^\rho \in P^\rho$, we have $m=P^\rho$ which must be a
self-conjugate line. This means that $(\ell,\ell')\in R^+$ implies that
$\ell\cap \ell' \not\in X$. It follows that any non self-conjugate line $\ell$
is contained in exactly $q(q-1)/2$ polar triangles. For further details and
background, see \cite[Section 7.3]{Hir98}

The abstract Hermitian unital $\mathcal{H}(q)$ is constructed from the set ${X}$ of self-conjugate points of $\rho$. The subsets cut out by the $(q+1)$-secants (not self-conjugate lines) form the set $B$ of blocks of $\mathcal{H}(q)$. Notice that we consider $\mathcal{H}(q)$ as an abstract unital, having a natural embedding in $\PG(2,q^2)$. The following proposition gives a characterization of the conjugate relation in terms of the abstract unital $\mathcal{H}(q)$ for $q$ even. 

\begin{proposition} \label{pr:even-hermitian}
Let $q$ be even, let $\rho$ be a Hermitian polarity of $\PG(2,q^2)$ and let $X$ be the set of self-conjugate points of $\rho$. Let $\ell_1,\ell_2$ be not self-conjugate lines and define the blocks $b_i=\ell_i\cap X$ of $\mathcal{H}(q)$, $i=1,2$. Then the following hold:
\begin{enumerate}[(i)]
\item If $\ell_1,\ell_2$ are conjugate, then $F_{\mathcal{H}(q)}(b_1,b_2)=b_3$, where $b_3=\ell_3\cap X$ with $\ell_3=(\ell_1\cap \ell_2)^\rho$. In other words, the blocks contained in a polar triangle form an embedded dual $3$-net of $\mathcal{H}(q)$. 
\item If $\ell_1,\ell_2$ are not conjugate then either $b_1\cap b_2\neq \emptyset$, or $|F_{\mathcal{H}(q)}(b_1,b_2)|=1$. 
\end{enumerate}
\end{proposition}
\begin{proof}
(i) Up to projective equivalence, we can assume that the matrix of $\rho$ is the identity. Since the unitary group $\PGU(3,q)$ acts transitively on $R^+$, we can assume $\ell_1:X_1=0$ and $\ell_2:X_2=0$. Then, $\ell_1\cap\ell_2=(0,0,1)$ and $\ell_3:X_3=0$. Let $\varepsilon$ be a $(q+1)$th root of unity in $\mathbb{F}_{q^2}$. The elements of $b_s=\ell_s\cap X$, $s=1,2,3$, have the form
\[A_i=(0,1,\varepsilon^i), \qquad B_j=(\varepsilon^j,0,1), \qquad C_k=(1,\varepsilon^k,0),\]
respectively, with $i,j,k=0,1,\ldots,q$. Since the points $A_i,B_j,C_k$ are collinear if and only if $\varepsilon^{i+j+k}=1$, we see that $A_i$ projects from $C_k$ to $B_{-i-k}$. In particular, $b_3\subseteq F_{\mathcal{H}(q)}(b_1,b_2)$, and equality holds by Theorem \ref{thm:embed-fptregular}. 

(ii) The case when $\ell_1,\ell_2$ are not conjugate and $b_1\cap b_2=\emptyset$ was elaborated in \cite[Section 2.2]{KSSz2018}. 
\end{proof}

\begin{remark} \label{rem:Hq_has_polartriangles}
Proposition \ref{pr:even-hermitian} shows that for $q$ even, $\mathcal{H}(q)$ has embedded dual $3$-nets. More precisely, any block of $\mathcal{H}(q)$ is contained in $q(q-1)/2$ polar triangles. The group of automorphisms of $\mathcal{H}(q)$ acts transitively on the set of embedded dual $3$-nets. 
\end{remark}

Let $\rho_0$ be a Hermitian polarity of the projective line $\PG(1,q^2)$. The set of self-conjugate points of $\rho_0$ forms a subline $\PG(1,q)$, cf. \cite[Lemma 6.2]{Hir98}. Let $\ell$ be a line of $\PG(2,q^2)$. A \textit{Baer subline} of $\ell$ is subset of size $q+1$, consisting of self-conjugate points of some Hermitian polarity $\rho$ of $\PG(2,q^2)$. Equivalently, a Baer subline $S$ is isomorphic to $\PG(1,q)$, and $S=\ell \cap \Pi$ for some line $\ell$ and a Baer subplane $\Pi$.

\begin{proposition} \label{prop:polartriables_baersublines}
Let $U=(X,B)$ be an abstract unital of order $q$, embedded in $\PG(2,q^2)$. Let $b_1,b_2,b_3$ form an embedded dual $3$-net. Then $b_1,b_2,b_3$ are Baer sublines. 
\end{proposition}
\begin{proof}
Let $\ell$ be the projective line containing $b_1$. By Theorem \ref{thm:embed-fptregular}, $C=\Persp_{b_2}(b_1)$ is a cyclic subgroup of order $q+1$, preserving $b_1$. Since $C$ is obtained using projections in $\PG(2,q^2)$, it is a subgroup of the projectivity group of $\ell$. By the arguments of \cite[Section 3]{KSSz2018} one shows that $b_1$ is a Baer subline  of $\ell$. 
\end{proof}

\begin{remark}
Let $q$ be even, and consider an \textit{arbitrary} embedding of the Hermitian unital $\mathcal{H}(q)$ in $\PG(2,q^2)$. By Remark \ref{rem:Hq_has_polartriangles} and Proposition \ref{prop:polartriables_baersublines}, all blocks correspond to Baer sublines of $\PG(2,q^2)$. Using the characterization of Hermitian curves from \cite{FainaKor1983,LPercsy}, this observation gives a simple proof of the main theorem in \cite{KSSz2018} in the even $q$ case.
\end{remark}

\section{Full points and dual $3$-nets of known small unitals}

In this section we present computational results on the structure of full points
of known small unitals. More precisely, we study the following classes of
abstract unitals of order at most $6$:

\begin{description}[align=right,labelwidth=3cm]
\item [Class BBT] 909 unitals of order 3 by Betten, Betten and Tonchev \cite{BBT03},
\item [Class KRC] 4466 unitals of order 3 with nontrivial automorphism groups by Kr\-\v{c}a\-di\-nac \cite{Kr06},
\item [Class KNP] 1777 unitals of order 4 by Kr\v{c}adinac, Naki\'c and Pav\v{c}evi\'c \cite{KNP11},
\item [Class BB] two cyclic unitals of order 4 and 6 by Bagchi and Bagchi \cite{BagchiBagchi1989}.
\end{description}

Notice that \KRC contains all abstract unitals of order $3$ with a nontrivial
automorphism group. As mentioned in \cite{Kr06}, 722 of the \BBT unitals appear
in \KRC. Moreover, the cyclic \BB unital of order $4$ is contained in \KNP. The
\BB unital of order $6$ has no full points, therefore we omit the \BB class
from the tables of this section. We access the libraries of small unitals and
carry out the computations using the GAP4 package \texttt{UnitalSZ}
\cite{MezNagy2018}.

\subsection{The number of full points and the structure of the group of perspectivities}

We only consider disjoint pairs of blocks admitting at least two full points as
for only one full point the perspectivitiy group is trivial. In Tables
\ref{tab:BBT-grs}, \ref{tab:KRC-grs} and \ref{tab:KNP-grs} we summarize the
existing number of full points, the structure of the group of perspectivities and
the number of unitals with such pairs for each library (\BBT, \KRC, \KNP).

\begin{table}[h!]
\caption{\label{tab:BBT-grs}\BBT unitals of order $3$}
\centering
\begin{tabular}{rlr}
\toprule
Full points & Group of perspectivities & Unitals\\
\midrule
2 & $C_{2}$ & 477\\
2 & $C_{3}$ & 94\\
2 & $C_{4}$ & 290\\
\bottomrule
\end{tabular}
\end{table}

\begin{table}[h!]
\caption{\label{tab:KRC-grs}\KRC unitals of order $3$}
\centering
\begin{tabular}{rlr}
\toprule
Full points & Group of perspectivities & Unitals\\
\midrule
2 & $C_{2}$ & 1015\\
2 & $C_{3}$ & 379\\
2 & $C_{4}$ & 897\\
3 & $S_{4}$ & 6\\
\bottomrule
\end{tabular}
\end{table}

\begin{table}[h!]
\caption{\label{tab:KNP-grs}\KNP unitals of order $4$}
\centering
\begin{tabular}{rlr}
\toprule
Full points & Group of perspectivities & Unitals\\
\midrule
2 & $C_{2}$ & 93\\
2 & $C_{4}$ & 71\\
2 & $C_{5}$ & 107\\
2 & $C_{6}$ & 5\\
3 & $A_{5}$ & 2\\
3 & $C_{2} \times C_{2}$ & 1\\
3 & $C_{4}$ & 32\\
3 & $C_{5}$ & 30\\
3 & $S_{5}$ & 3\\
4 & $C_{5}$ & 8\\
5 & $C_{5}$ & 165\\
6 & $C_{5} \rtimes C_{4}$ & 72\\
6 & $D_{10}$ & 53\\
\bottomrule
\end{tabular}
\end{table}

\subsection{The structure of the full points}

The structure of the full points is only interesting when there is at least 3
of them, hence the \BBT unitals are out of our scope. Even the case of 3 full
points is simple: they are either contained in a block or not. As \KRC unitals
admit at most 3 full points, we are only interested in the \KNP unitals.

The computation in \cite{MezNagy2018} showed that if there are 4 or 5 full
points (in the case of disjoint blocks) then either the whole set of full
points is contained in a single block, or no three points are collinear.
Similarly in the case of 6 full points either 5 of the full points form a block
or no 3 of them are collinear. Now by ``collinear'' we mean that the points
form a subset of some block of the unital.

\subsection{Unitals with large full point sets}

Let us denote by $\Omega$ the subset of unitals with at least one \textit{large}
full point set, that is, $|F_U(b_1,b_2)|\geq 3$ for a pair $(b_1,b_2)$ of
disjoint blocks. We have seen that $\Omega$ is the empty set for \BBT unitals.
By Table \ref{tab:KRC-grs}, $|\Omega|=6$ for \KRC unitals. Hence, the interesting case is the \KNP library, where the size of $\Omega$ is 206.
In Table \ref{tab:KNP-large} we present the number of \KNP unitals with some
restrictions on the structure of full points. Clearly $A \subseteq B$, $C
\subseteq \overline{B}$ and $\Omega = B \cup \overline{B}$.

\begin{table}[h!]
\caption{\label{tab:KNP-large}\KNP unitals with large full point sets}
\centering
\begin{tabular}{rlr}
\toprule
	set & property & cardinality\\
\midrule
$\Omega$ & \text{at least one \textit{large} full point set} & 206 \\
$A$ & \text{all large full point sets form a block} & 74 \\
$B$ & \text{all large full point sets are contained in a block} & 80 \\
$\overline{B}$ & \text{some large full point sets are not contained a block} & 126\\
$C$ & \text{no large full point set is contained in a block} & 1\\
\bottomrule
\end{tabular}
\end{table}


\subsection{Full point regularity}

In Table \ref{tab:FullPtReg} one sees how many of the unitals of the different
libraries are full point regular (FPR) and strongly full point regular (SFPR).
In fact, if a unital is not strongly full point regular then is not embeddable
into $\PG(2, q^2)$. Hence 94 \BBT unitals, 385 \KRC unitals and 195 \KNP
unitals are definitely not embeddable into $\PG(2, q^2)$. Notice that \cite{BBPW14}
proves a much stronger result, where the authors show that there are just two 
orbits of unitals in $\PG(2,16)$, containing the Hermitian unitals and 
Buekenhout–Metz unitals, respectively. 

\begin{table}[h!]
\caption{\label{tab:FullPtReg}Full point regularity}
\centering
\begin{tabular}{lrrr}
\toprule
Library & Unitals & FPR & SFPR\\
\midrule
\BBT & 909 & 815 & 815\\
\KRC & 4466 & 4081 & 4081\\
\KNP & 1777 & 1586 & 1582\\
\bottomrule
\end{tabular}
\end{table}

\subsection{Embedded dual 3-nets}

By Proposition \ref{prop:dual_k_nets}(ii), one can find embedded dual 3-nets only among the \KNP unitals. The computation shows us that the latin squares constructed from the dual 3-nets are always of cyclic type, namely, any embedded dual 3-net is cyclic in the \KNP library.
However, we constructed a new unital of order 4 with a non-cyclic embedded dual 3-net, cf. Appendix \ref{sec:app}.

\appendix

\section{Unital of order $4$ with non-cyclic embedded dual $3$-net} \label{sec:app}

\begin{lstlisting}[language=GAP,basicstyle=\scriptsize]
LoadPackage("UnitalSZ");

bls:=[[1,2,55,64,65], [1,3,32,46,63], [1,4,7,34,45], [1,5,11,31,44],
    [1,6,12,19,54], [1,8,38,47,50], [1,9,24,27,40], [1,10,20,48,53],
    [1,13,17,49,57], [1,14,15,16,29], [1,18,33,43,58], [1,21,23,25,37],
    [1,22,51,56,60], [1,26,30,39,52], [1,28,36,41,62], [1,35,42,59,61],
    [2,3,6,30,58], [2,4,14,54,60], [2,5,29,46,47], [2,7,13,48,59],
    [2,8,34,37,40], [2,9,10,18,31], [2,11,19,32,52], [2,12,20,50,57],
    [2,15,21,43,62], [2,16,23,27,28], [2,17,33,45,61], [2,22,24,25,26],
    [2,35,38,39,41], [2,36,49,53,56], [2,42,44,51,63], [3,4,19,23,33],
    [3,5,10,39,59], [3,7,22,49,52], [3,8,14,48,65], [3,9,25,29,60],
    [3,11,15,20,34], [3,12,13,16,61], [3,17,24,28,44], [3,18,47,53,57],
    [3,21,36,40,42], [3,26,37,38,43], [3,27,35,56,64], [3,31,45,55,62],
    [3,41,50,51,54], [4,5,41,52,53], [4,6,26,31,47], [4,8,16,36,57],
    [4,9,56,58,59], [4,10,28,46,65], [4,11,21,50,64], [4,12,35,44,62],
    [4,13,30,32,51], [4,15,37,39,61], [4,17,18,20,25], [4,22,27,38,48],
    [4,24,42,49,55], [4,29,40,43,63], [5,6,7,32,37], [5,8,42,54,58],
    [5,9,12,17,51], [5,13,27,36,63], [5,14,22,61,62], [5,15,25,40,49],
    [5,16,19,20,26], [5,18,21,28,38], [5,23,30,55,60], [5,24,33,48,64],
    [5,34,35,57,65], [5,43,45,50,56], [6,8,56,62,63], [6,9,21,61,65],
    [6,10,14,40,41], [6,11,25,43,51], [6,13,38,44,55], [6,15,42,46,57],
    [6,16,22,34,64], [6,17,23,36,52], [6,18,48,49,60], [6,20,28,45,59],
    [6,24,35,50,53], [6,27,29,33,39], [7,8,20,24,51], [7,9,41,63,64],
    [7,10,11,42,60], [7,12,15,55,56], [7,14,23,26,35], [7,16,44,46,53],
    [7,17,29,38,62], [7,18,19,39,50], [7,21,27,31,57], [7,25,47,58,65],
    [7,28,30,33,40], [7,36,43,54,61], [8,9,11,13,46], [8,10,12,45,52],
    [8,15,18,27,59], [8,17,21,35,60], [8,19,43,49,64], [8,22,29,30,53],
    [8,23,32,39,44], [8,25,31,33,41], [8,26,28,55,61], [9,14,52,55,57],
    [9,15,19,28,53], [9,16,35,43,47], [9,20,22,36,39], [9,23,48,50,62],
    [9,26,32,33,42], [9,30,34,38,54], [9,37,44,45,49], [10,13,23,34,43],
    [10,15,17,30,64], [10,16,21,32,56], [10,19,25,35,55], [10,22,33,54,57],
    [10,24,36,37,47], [10,26,27,51,62], [10,29,44,50,61], [10,38,49,58,63],
    [11,12,33,38,59], [11,14,39,47,56], [11,16,18,54,62], [11,17,22,41,65],
    [11,23,24,29,57], [11,26,36,45,48], [11,27,30,49,61], [11,28,35,37,63],
    [11,40,53,55,58], [12,14,24,30,43], [12,18,23,42,65], [12,21,26,41,49],
    [12,22,28,32,47], [12,25,34,48,63], [12,27,37,53,60], [12,29,31,36,58],
    [12,39,40,46,64], [13,14,21,33,53], [13,15,41,45,47], [13,18,26,29,64],
    [13,19,24,31,56], [13,20,35,52,58], [13,22,37,42,50], [13,25,28,39,54],
    [13,40,60,62,65], [14,17,19,59,63], [14,18,37,46,51], [14,20,31,38,42],
    [14,25,36,44,64], [14,27,32,45,58], [14,28,34,49,50], [15,22,23,31,63],
    [15,24,32,38,65], [15,26,50,58,60], [15,33,35,36,51], [15,44,48,52,54],
    [16,17,37,48,58], [16,24,41,59,60], [16,25,30,42,45], [16,31,39,49,65],
    [16,33,50,55,63], [16,38,40,51,52], [17,26,34,46,56], [17,27,47,54,55],
    [17,31,32,40,50], [17,39,42,43,53], [18,22,35,40,45], [18,24,52,61,63],
    [18,30,41,44,56], [18,32,34,36,55], [19,21,22,44,58], [19,27,34,41,42],
    [19,29,45,51,65], [19,30,37,57,62], [19,36,38,46,60], [19,40,47,48,61],
    [20,21,30,47,63], [20,23,40,54,56], [20,27,43,44,65], [20,29,37,41,55],
    [20,32,60,61,64], [20,33,46,49,62], [21,24,45,46,54], [21,29,34,52,59],
    [21,39,48,51,55], [22,43,46,55,59], [23,38,45,53,64], [23,41,46,58,61],
    [23,47,49,51,59], [24,34,39,58,62], [25,27,46,50,52], [25,32,53,59,62],
    [25,38,56,57,61], [26,40,44,57,59], [26,53,54,63,65], [28,29,42,48,56],
    [28,31,43,52,60], [28,51,57,58,64], [29,32,35,49,54], [30,31,35,46,48],
    [30,36,50,59,65], [31,34,51,53,61], [31,37,54,59,64], [32,41,43,48,57],
    [33,34,44,47,60], [33,37,52,56,65], [39,45,57,60,63], [42,47,52,62,64]
];;

u:=AbstractUnitalByDesignBlocks(bls);
t:=BlocksOfUnital(u){[1,33,200]};
StructureDescription(PerspectivityGroupOfUnitalsBlocks(u,t[1],t[2],t[3]));
\end{lstlisting}

\end{document}